\documentclass{elsarticle}

\makeatletter
\def\ps@pprintTitle{%
 \let\@oddhead\@empty
 \let\@evenhead\@empty
 \def\@oddfoot{}%
 \let\@evenfoot\@oddfoot}
\makeatother
\usepackage{amsthm}
\usepackage{graphicx}
\usepackage{hyperref}
\usepackage{amsfonts}
\usepackage{epstopdf}
\usepackage{algorithmic}
\usepackage{amssymb,amsmath,epsfig,multirow}
\usepackage{mathtools} 
\usepackage{color}
\usepackage{lineno,hyperref}
\usepackage{rotating}
\usepackage{makeidx}
\usepackage{lscape}
\usepackage{float}
\usepackage{epsfig}
\usepackage{amsmath,amssymb}
\usepackage{enumitem}
\usepackage{epic,eepic}
\usepackage{overpic}
\usepackage{color}
\usepackage{amsfonts}
\usepackage{graphicx}
\usepackage{enumitem}
\usepackage{pgf,tikz,pgfplots}
\pgfplotsset{compat=1.15}
\usepackage{mathrsfs}

\usetikzlibrary{arrows}

\newtheorem{thm}{Theorem}
\numberwithin{thm}{section}

\newtheorem{prop}[thm]{Proposition}

\newcommand{\ZZ}{\mathbb{Z}}
\newcommand{\RR}{\mathbb{R}}

\begin{document}
\title{Behaviour Preserving Extensions of Univariate and Bivariate Functions}
\author{David Levin}
\ead{levindd@gmail.com}
\date{Received: date / Accepted: date}
\address{School of Mathematical Sciences. Tel-Aviv University. Tel-Aviv, Israel}

\begin{abstract}
Given function values on a domain $D_0$, possibly with noise, we examine the
possibility of extending the function to a larger domain $D$, $D_0\subset D$.
In addition to smoothness at the boundary of $D_0$, the extension on
$D\setminus D_0$ should also inherit behavioral trends of the function on
$D_0$, such as growth and decay or even oscillations. The approach chosen here
is based upon the framework of linear models, univariate or bivariate, with
constant or varying coefficients.

\end{abstract}

\maketitle

\section{Introduction}
R. Prony \cite{Prony} has suggested in 1795 to model a data sequence $\zeta_n$
by a of complex exponentials, and he has already shown that such a model is
equivalent to finding a linear model with constant coefficients for the sequence's evolution. D. Shanks \cite{Shanks} has revived this approach,
used it for convergence acceleration, and has shown its relation to Pad\'e
approximation. Prony's method for the decomposition of a signal with a finite
number of complex exponentials is a powerful tool for signal analysis. It has
been studied by many authors, who made valuable contributions to the practical
application of the method, e.g., by Osborne and Smyth \cite{Osborne}.

Two-dimensional (2-D) linear prediction models with constant coefficients were utilized in \cite{Levin1980} within the framework of double series and bivariate Padé approximation. In the 2-D case, the equivalence between linear prediction and exponential fitting no longer holds. Consequently, the class of double sequences that satisfy 2-D linear models with constant coefficients is generally broader than simple sums of exponentials.

1-D linear models with varying coefficients have been studied in
\cite{Levin1973} and in \cite{LevinSidi} who introduced new sequence
transformations that are efficient for a wide class of sequences, wider than
those that are defined by sums of exponentials. For a complete overview see
\cite{SidiBook}.

Prony's method is closely related to least squares linear prediction algorithms and is highly effective for sequence extension or extrapolation. Extending a function presents a similar challenge, with the added constraint that the extension must also maintain smoothness.

 In this paper, we first demonstrate how a univariate linear prediction model with constant coefficients can be effectively used for smooth functional extension without the need to solve for the exponents, as required in Prony's method. Moreover, we show that this approach can be extended to higher-dimensional prediction. In Section 5, we further refine the method and propose a spline-based model that significantly reduces computational complexity.

\section{Univariate case - from a linear model to extension}
Given a function on a domain $D_0$, the ideal information for extending the
function into a larger domain would be the differential equation that the
function satisfies on $D_0$. If the differential equation is simple, e.g.,
with constant coefficients, it can be extended beyond $D_0$, and the extension
of the function may be defined by solving the differential equation with
proper initial or boundary conditions. If the function is known only on a
discrete set of points in $D_0$, say on a uniform grid, we may hope to find a
difference equation that the data values satisfy on the grid. Knowing the
difference equation may serve as a starting point for both approximation the
function on $D_0$ and extending it beyond $D_0$, as explained below. It is
important to note that a univariate function is a sum of exponentials if and
only if it satisfies a linear differential equation with constant
coefficients. Also, the values of a univariate function on a regular grid is a
sum of exponentials if and only if they satisfy a linear difference equation
with constant coefficients. This equivalence does not hold in 2-D. The case of
linear differential or difference equations with varying coefficients is
studied in \cite{LevinSidi} for the generation of the $d$ and the $D$
transformations for the efficient extrapolation of infinite series and
integrals. It is shown there that the class of functions satisfying such
equations covers most of the special functions, and their combinations.

\subsection{Extracting linear prediction models}
Let $D_0=[a,b]$ and consider data sets $\{x_i,f_i\}$, where
$\{f_i\}$ are function values (possibly with noise) at equidistant
points $x_i=a+ih, i=0,...,N$, $h=(b-a)/N$. We would like to find a
difference equation, or a linear prediction model, by which we can
extend the function for $x>b$. Recalling our declared goal, we would
like the extension to carry along the characteristic behavior of the
function within the interval $[a,b]$. Since $h$ may be very small,
difference relation on a sequence of data values at distance $h$
cannot catch the global behavior of $f$ on $[a,b]$. Also, in
particular, in the presence of noise, the problem of finding a difference
equation satisfied by the given data may be quite unstable. Let
$d=nh$, we quest for a linear prediction model of order $m$
satisfied by all the data sequences of mesh size $d$,
$\{f_{i+(j-1)n}\}_{j=1}^{[N/n-1]}$, $i=0,...,n$. As we shall see
below, the value $d$ determines, by Nyquist sampling theory, the
frequencies that can be reconstructed by the prediction model. We
consider linear prediction models of the form:
\begin{equation}\label{eq:pred1}
[1+q_{m+1}u(x_i)]f_i=\sum_{k=1}^m [p_k+q_ku(x_i)] f_{i-(m-k+1)n}.
\end{equation}
Typical choices of the function $u$ would be:
\begin{enumerate}
\item $u\equiv 0$ for a model with constant coefficients.
\item $u(x)=x$ for a linearly varying model.
\item $u(x)=\frac{1}{x+\alpha}$ for a model with rational
variation.
\end{enumerate}

Now we may use a standard way of defining an approximate prediction
model by least-squares fit, as follows:

We look for model coefficients $P=\{p_k\}_{k=1}^m$ and
$Q=\{q_k\}_{k=1}^{m+1}$ such that
\begin{equation}\label{eq:ls1}
I_1(P,Q)=\sum_{i=mn}^N[[1+q_{m+1}u(x_i)]f_i-\sum_{k=1}^m
[p_k+q_ku(x_i)] f_{i-(m-k+1)n}]^2\ \to\ min.
\end{equation}

In Section 4 we discuss other options for extracting the model. In
particular, we consider improving the numerical stability by
minimizing
\begin{equation}\label{eq:ls4}
I_1(P,Q)+\nu \sum_{k=1}^m p_k^2+\mu\sum_{k=1}^{m+1} q_k^2.
\end{equation}
In Section 5 we use linear models for the bivariate extension
problem.

\vfill\eject
\section{Approximation and extension algorithms}

\subsection{Univariate models with constant coefficients}\label{sec1}

Let us first discuss a linear model with constant coefficients.
I.e., we would like to approximate the data on $[a,b]$, and extend
it beyond $b$, using the linear model
\begin{equation}\label{eq:pred2}
g_i=\sum_{k=1}^m p_k g_{i-(m-k+1)n}.
\end{equation}
Such a model takes us back to Prony's method, i.e., approximation by
a sum exponentials. Let $\{\lambda_j\}_{j=1}^m$ be the roots of the
polynomial
\begin{equation}\label{eq:char}
p(\lambda)=\sum_{k=1}^m p_k \lambda^{k-1} -\lambda^m.
\end{equation}
We also assume, for simplicity, that all the roots of $p$ are
simple. Then, all sequences satisfying (\ref{eq:pred2}) are of the
form
\begin{equation}\label{eq:expsum1}
g_{rn+i}=\sum_{j=1}^m c^{(i)}_j \lambda_j^r,\ r\in \ZZ.
\end{equation}
with coefficients $\{c^{(i)}_j\}$ depending on $i$. We recall that
the value $g_{rn+i}$ is attached to the point $x=a+(rn+i)h=a+rd+ih$,
and we would like to define a smooth function $g$ such that
$g(a+rd+ih)=g_{rn+i}$. W.l.o.g., we may assume that $d=1$, and
obtain the relation
\begin{equation}\label{eq:expsum2}
g(a+r+ih)=\sum_{j=1}^m c^{(i)}_j \lambda_j^r=\sum_{j=1}^m {\tilde
c}^{(i)}_j \lambda_j^{a+r+ih},\ r\in \ZZ.
\end{equation}
The evident way of defining a smooth $g$ on $\RR$ is to make the
coefficients independent on $i$, i.e., ${\tilde c}^{(i)}_j={\tilde
c}_j$ for any $i$. Then, $g$ is simply
\begin{equation}\label{eq:expsum3}
g(x)=\sum_{j=1}^m {\tilde c}_j\lambda_j^x,\ \ \ x\in \RR.
\end{equation}
Other ways for constructing a smooth $g$ satisfying the model will
be presented in the following sections.

\bigskip
{\bf The approximation extension algorithm for linear model with
constant coefficients:}

\begin{enumerate}
\item Find the model coefficients by (\ref{eq:pred1}) with $u=0$.
\item Find the exponents $\{\lambda_j\}_{j=1}^m$ as the roots of the
characteristic polynomial (\ref{eq:char}).
\item Define the approximation on $[a,b]$ and the extension by
(\ref{eq:expsum3}) where the coefficients $\{{\tilde c}_j\}$ are
obtained by least-squares approximation to the given data on
$[a,b]$.
\end{enumerate}

There are some technical issues to deal with in the above algorithm. One is
the case of multiple roots in step 2, and another is the problem of complex
approximation to a real function in step 3. The last issue can be resolved by
replacing the basis functions in step 3 above by an independent subset of
$\{Re(\lambda_j^x),Im(\lambda_j^x)\}_{j=1}^m$.  More problematic is the issue
of choosing the right order $m$ for the model. If $m$ is too small the model
cannot approximate the data, and if $m$ is too large some of the extra
resulting exponents may introduce highly oscillatory behavior or another type
of instability. Altogether, as shown in examples 1-2 below, the above
algorithm works quite nicely for noisy data of functions for which can be well
approximated by sums of exponentials. However, the above approach cannot be
applied to models with varying coefficients, and in the bivariate case it is
not suitable even for models with constant coefficients. Therefore, the main
purpose of this paper would be to suggest more general approximation-extension
algorithms that are applicable to those cases as well.

\bigskip
{\bf Examples of approximation extension using exponential's
fitting:}

As an example we considered extension of the function:
\begin{equation}\label{eq:f1}
f_1(x)=.8^x-cos(x)+2sin(2x)+\frac{1}{x+1}\ \ \ x\in [0,7],
\end{equation}
to $[0,14]$. We first demonstrate the results for non-noisy data. The
parameters used are $n=50$, $h=0.02$, and a model of the form (\ref{eq:pred2})
with $m=6$. The resulting exponents are:
$$\{\lambda_j\}=\{0.061818, 0.772124,-0.416977\pm 0.908787i,0.520298\pm 0.852041i\}.$$
Note that $0.8$, $e^i=0.540302\pm 0.841471i$ and $e^{2i}=-0.416149\pm
0.909297i$ are the exact exponentials constituting $f_1$. However, since $f_1$
is not a 'pure' sum of exponentials, we do not get them exactly. In Figure 1
we plot the function $f_1$ together with the reconstructed
approximation-extension $g$ on $[0,14]$ (defined by (\ref{eq:expsum3})), using
function values only on $[0,7]$.

Next we consider the same test function $f_1$, measured at the same points in
$[0,7]$, but with an added noise, randomly distributed in $[-0.2,0.2]$. Here
the resulting exponents of a model of the same size, $m=6$, turn to be
$$\{\lambda_j\}=\{0.273076,-0.864675,-0.414134\pm 0.908059i,0.542686\pm 0.562524i\}$$
In Figure 2 we plot the reconstructed approximation-extension $g$ on $[0,14]$
together with the data on $[0,7]$ used in the algorithm, and the exact
function $f_1$ on $(7,14]$. We note that the complex exponents are not so
sensitive to the noise, but the real frequencies are quite different. Yet, the
approximation in $[0,7]$ is good, and the extension is also reasonable.
\vfill\eject

\begin{figure}[hb]
  \centering
\includegraphics[width=3in]{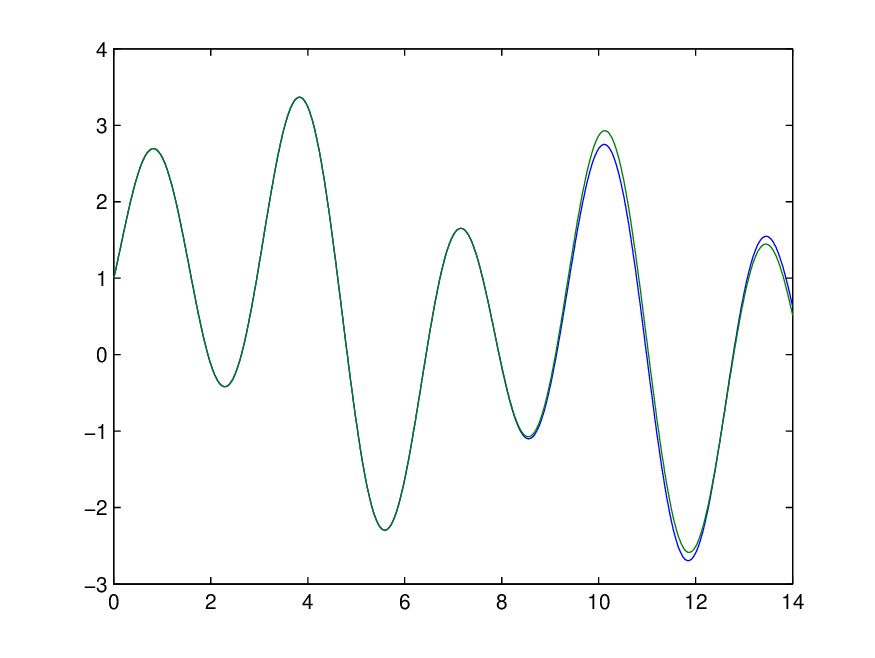}
  \caption[example1]{Approximation to non-noisy data on $[0,7]$ and extension on $[0,14]$}
\end{figure}
\begin{figure}[hb]
  \centering
  \includegraphics[width=3in]{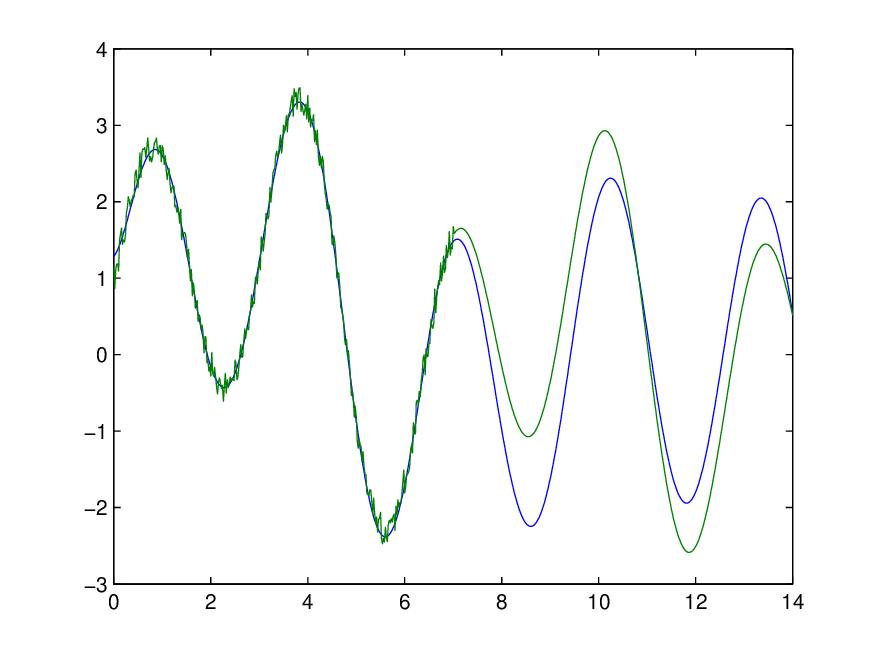}
  \caption[example2]{Approximation to noisy data on $[0,7]$ and extension on $[0,14]$}
\end{figure}

\subsection{Univariate models with varying coefficients}
Unlike the above Prony's type method, the algorithm presented below
does not rely on finding the general solution of sequences
satisfying the model. Hence, in principle, it is applicable even for
non linear models with varying coefficients. The method is based
upon joining together all the required elements into one objective
functional, and minimizing this functional within the space of all
sequences satisfying the model. We demonstrate the approach via
linear models with constant or varying coefficients.

Let us denote by $g\in M$ a sequence $g=\{g_i\}_{n_0\le i \le n_1}$,
$n_0\le 0$ and $n_1\ge N$, satisfying a model $M$. We would like to
find a sequence $g\in M$ such that:

 {\bf a.} $\{g_i\}_{i=0}^N$ approximates the given data
sequence $\{f_i\}_{i=0}^N$.

 {\bf b.} $g$ is smooth.

Requirement {\bf a.} is reflected in the functional
\begin{equation}\label{eq:lsg}
E(g)=\sum_{i=0}^N[f_i-g_i]^2.
\end{equation}
while the smoothness requirement {\bf b.} is represented by
\begin{equation}\label{eq:spg}
S_p(g)=\sum_{n_0\le i \le n_1-p}[\Delta^pg_i]^2,
\end{equation}
where $\Delta$ is the ordinary difference operator.

\bigskip
{\bf The approximation extension algorithm for general models:}

\begin{enumerate}
\item Find an appropriate model $M$ for the data $\{f_i\}_{i=0}^N$.
\item Define the approximation on $\{x_i\}_{i=0}^N$ and its extension as the sequence
$g=\{g_i\}_{n_0\le i \le n_1}$ minimizing the functional
\begin{equation}\label{eq:F_p}
F_p(g)=S_p(g)+\mu E(g),\ \  g\in M.
\end{equation}
\end{enumerate}

{\bf Discussion:} The new element here is the inclusion of the
smoothing functional $S_p$. Such a functional is a very common tool
in Computer Aided Geometric Design, where it is used for smooth
filling of holes in a surface. In approximation theory $S_p$ is
viewed as a regularization functional. Let us explain its special
role in our context; In the case of a model with varying
coefficients, if the coefficients vary smoothly, each sequence
obeying the model may be smooth, in some sense. However, let us
recall that the model connects points that are $d$ distant apart,
or, equivalently, $n$ indices apart, from each other (see
(\ref{eq:pred2})). Therefore, within $g\in M$ there may be $n$
smooth independent subsequences satisfying the model. The first role
of the functional $S_p$ is to force those $n$ subsequences of $g$ to
unite into one smooth sequence. Furthermore, even in the case of
models with constant coefficients, the space of sequences satisfying
the model may include parasitic highly oscillatory sequences. The
second role of smoothing functional $S_p$ is to invalidate these
parasite components in $M$. The parameter $\mu$ determines a balance
between the functionals $S_p$ and $E$. In the examples below we
discuss the effects of the parameter $\mu$ and the order $p$ of the
difference operator in $S_p$.

One may also argue that it is enough to take care of the smoothness
of $\{g_i\}_{i=0}^N$, and to continue the sequence by the model.
However, this approach has been found to be unstable, and moreover,
as explained in Section 5, this approach cannot work in 2-D.

\bigskip
{\bf Examples of approximation-extension using a model
and a smoothing functional:}
\bigskip
Here we repeat the example with $f_1$ defined in (\ref{eq:f1}), measured at the
same points in $[0,7]$, with an added noise randomly distributed in
$[-0.2,0.2]$. The model is also of the same size, $m=6$, with constant
coefficients, but the reconstruction is computed by minimization of $F_2$
defined in (\ref{eq:F_p}) with $\mu=0.001$.

\begin{figure}[hb]
  \centering
  \includegraphics[width=3in]{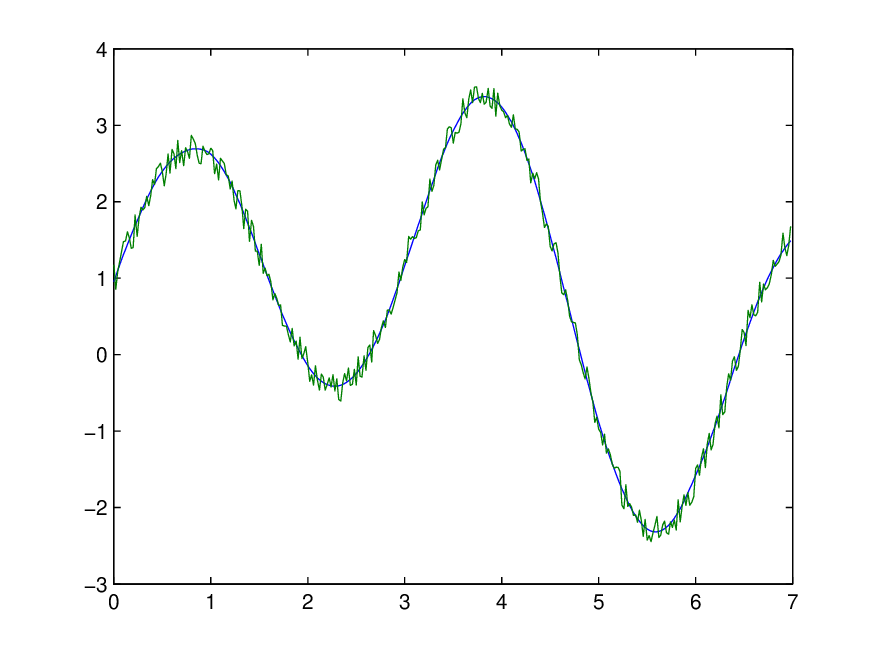}
  \caption[example1]{Approximation to the non-noisy data on $[0,7]$.}
\end{figure}

\begin{figure}[hb]
  \centering
  \includegraphics[width=3in]{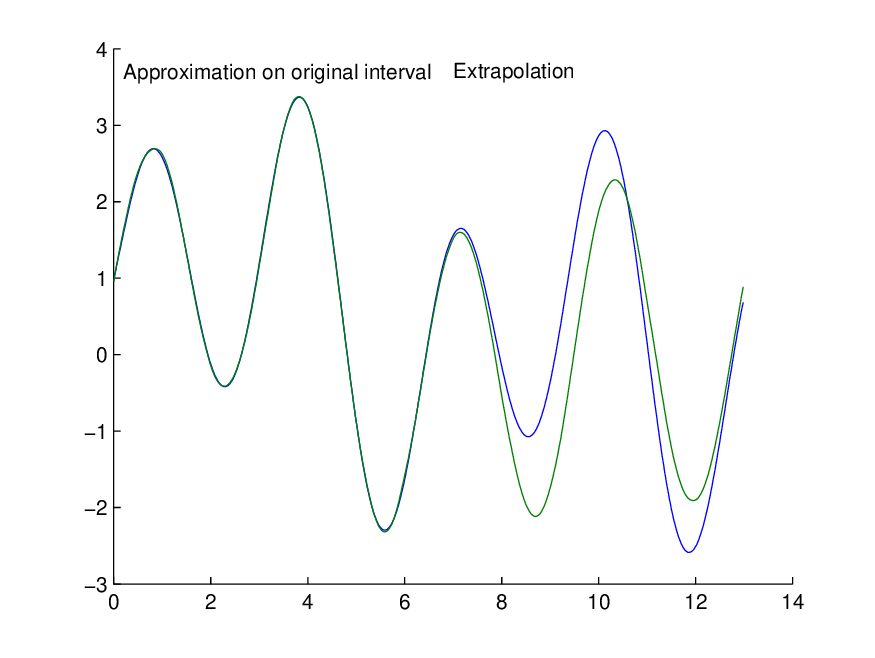}
  \caption[example2]{Approximation-extension on $[0,14]$}
\end{figure}

\subsection{The scope of linear models with varying coefficients}\label{sec2}

D. Shanks \cite{Shanks} investigated the use of exponentials fitting to a
sequence as a tool for convergence acceleration. It is also shown there that
fitting a linear model with constant coefficients to the terms of a power
series leads to Pad\'e approximations. The use of linear models with varying
coefficients has been studied in \cite{Levin1973} and in \cite{LevinSidi} who
introduced new sequence transformations that are efficient for a wide class
of sequences, wider than those that are defined by sums of exponentials. In
\cite{LevinSidi} we can find examples of classes of sequences $\{a_k\}$ that
satisfy a linear model with coefficients that has asymptotic expansion in
inverse powers of $k$, as $n\to\infty$. We refer to models of general order
$m$ of the form:
\begin{equation}\label{eq:modela}
a_{k+m+1}=\sum_{i=1}^m p_i(k) a_{k+i},
\end{equation}
where
\begin{equation}\label{eq:pi}
p_i(k)\sim k^{r_i}\sum_{j=0}^\infty p_{i,j}k^{-j},\ \ \ as\ \
k\to\infty,\ \ r_i\in \ZZ.
\end{equation}
As a simple example consider the sequence $a_k=k^{\alpha}e^{ck}$.
Obviously,
\begin{equation}\label{eq:b1}
a_{k+1}=e^c(1-\frac{1}{k+1})^\alpha a_k,
\end{equation}
that is a linear model, with coefficients that vary with $k$. For any $c$ and
$\alpha$, the coefficients has asymptotic expansion in inverse powers of $k$,
as $k\to\infty$. Following \cite{LevinSidi} we denote by $B^{(m)}$ the class
of sequences satisfying a model of order $m$ of the form (\ref{eq:modela})
with coefficients that has asymptotic expansions of the form (\ref{eq:pi}).

Similar to the case of exponentials, the following algebraic rules hold
(\cite{LevinSidi}):

Let $\{a_k\}\in B^{(m_1)}$ and $\{b_k\}\in B^{(m_2)}$ then
\begin{equation}\label{eq:ls80}
\{a_k+b_k\}\in B^{(m_1+m_2)},\ \ \{a_kb_k\}\in B^{(m_1m_2)}.
\end{equation}

For example, by the above properties we can analyze sequences of
equidistant evaluations of Bessel functions, concluding that
$\{a_k\}=\{J_\nu(kd)\}\in B^{(2)}$ for any order $\nu$ and for any
$d$. Altogether, we recall here the rich family of sequences
satisfying linear models of type $B^{(m)}$. In this paper we report
the use of a restricted sub-class of $B^{(m)}$, namely models of the
form (\ref{eq:pred1}) with $u(x)=x$ or $u(x)=\frac{1}{x+\alpha}$.

\bigskip
{\bf Examples of approximation-extension using a model with rational
coefficients:}
\bigskip
As an example we considered extension of the function:
\begin{equation}\label{eq:f2}
f_2(x)=\frac{5cos(2x)}{x^2+1}+x^{1.5}sin(x),\ \ \ x\in [0,7],
\end{equation}
to $[0,14]$. Here again $f_2$ is measured with a noise, randomly distributed
in $[-0.2,0.2]$. The parameters used are $n=100$, $h=0.01$, and a model of the
form (\ref{eq:pred1}) with $u(x)=\frac{1}{x+\alpha}$ and $m=6$. The
reconstruction is computed by minimization of $F_2$ defined in (\ref{eq:F_p})
with $\mu=0.0001$ ($\mu\sim h^2$).

\vfill\eject
\begin{figure}[hb]
  \centering
  \includegraphics[width=3in]{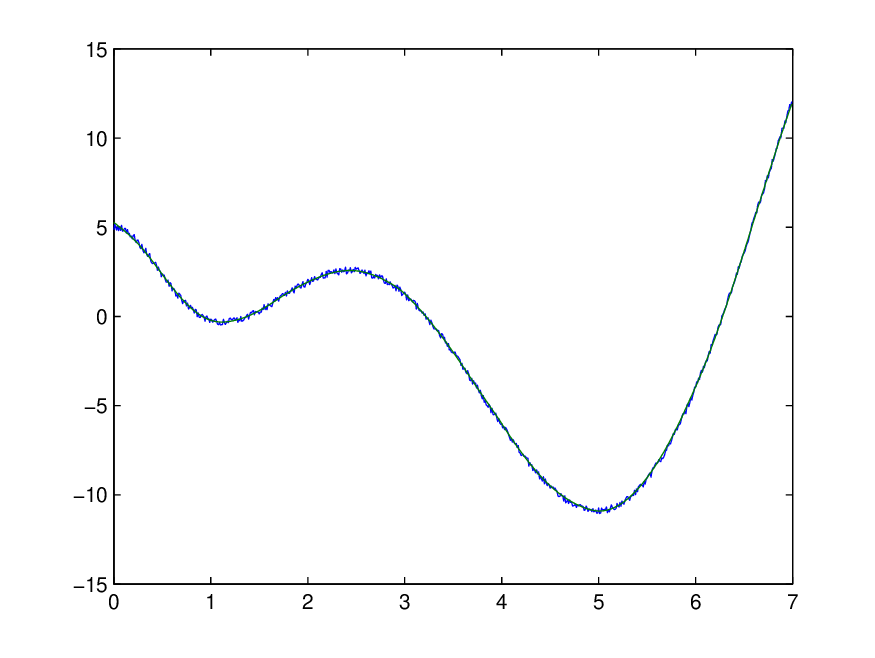}
  \caption[example1]{Rational coefficients model: Approximation to the non-noisy data on $[0,7]$.}
\end{figure}

\begin{figure}[hb]
  \centering
  \includegraphics[width=3in]{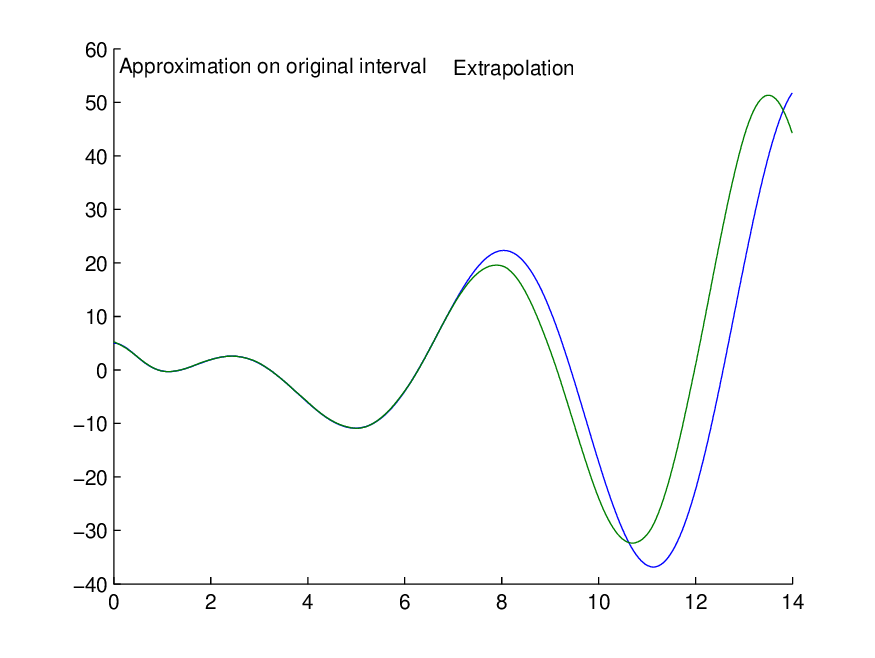}
  \caption[example2]{Rational coefficients model: Approximation-extension on $[0,14]$}
\end{figure}

\bigskip
{\bf Remarks:}
\begin{enumerate}
\item Using the above characterization rules, the function (\ref{eq:f2}) is in $B^{(4)}$, and it cannot
satisfy a linear model with constant coefficients. Our model, with
rational coefficients of low degree, is also not ideal here, but it
performs a little better that a model with constant coefficients.
\item The right choice of the parameter $\mu$ is important. Choosing
$\mu$ too big results in a bad approximation to the data, while a
too small $\mu$ yields non-smooth approximation. The rule
$\mu=0.0001$ $\mu\sim h^2$ works well.
\end{enumerate}

\vfill\eject
\section{The bivariate case - from a linear model to smooth extension}

W.L.O.G, let $D_0=[a,b]\times [a,b]$ and consider data sets
$\{(x_i,y_j),f_{i,j}\}$, where $\{f_{i,j}\}$ are function values
(possibly with noise) on a square mesh $(x_i,y_j)=(a+ih,a+jh),
i,j=0,...,N$, $h=(b-a)/N$. We would like to find a difference
equation, or a linear prediction model, by which we can extend the
function into $D=[c,d]\times [c,d]$, $c<a$, $d>b$. As in the
univariate case, we start by finding a linear model of size $m\times
m$ satisfied by all the data sequences of mesh size $d=nh$,
$\{f_{i+(k-1)n,j+(\ell-1)n}\}_{k,\ell=1}^{[N/n-1]}$, $i,j=0,...,n$.
Denoting $K=\{(k,\ell): k,\ell=1,...,m\}$, we consider linear models
of the form:
\begin{equation}\label{eq:model2D}
\sum_{(k,\ell)\in K} p_{k,\ell} f_{i+(k-1)n,j+(\ell-1)n}=0,
\end{equation}
with the normalization $p_{m,m}=1$.

The first step in the extension algorithm is finding an approximate
model $M$ for the given data. As in the 1-D case we do it by
least-squares minimization. Define $K^-=K\setminus \{(m,m)\}$We look
for model coefficients $P=\{p_{(k,\ell)}\}_{(k,\ell)\in K^-}$ and
such that
\begin{equation}\label{eq:ls2}
I_2(P)=\sum_{(i,j)}\big[[\sum_{(k,\ell)\in K^-} p_{k,\ell}
f_{i+(k-1)n,j+(\ell-1)n}]+f_{i+(m-1)n,j+(m-1)n}\big]^2 \to\ min.
\end{equation}

As remarked in the introduction, 2-D Linear prediction models with
constant coefficients describe function spaces that are much richer
than sums of exponentials. In fact the space of functions satisfying
a given 2-D model is usually of infinite dimension. Hence, the
method of fitting exponentials cannot work here, and we use the
above approximation+smoothing algorithm, that does not rely on
finding the general solution of sequences satisfying the model.

Let us denote by $g\in M$ a sequence $g=\{g_{i,j}\}_{n_0\le i,j \le
n_1}$, $n_0\le 0$ and $n_1\ge N$, satisfying a model $M$. We would
like to find a sequence $g\in M$ such that:

 {\bf a.} $\{g_{i,j}\}_{i,j=0}^N$ approximates the given data
sequence $\{f_{i,j}\}_{i,j=0}^N$.

 {\bf b.} $g$ is smooth.

Requirement {\bf a.} is reflected in the functional
\begin{equation}\label{eq:lsg2}
E(g)=\sum_{i,j=0}^N[f_{i,j}-g_{i,j}]^2.
\end{equation}
while the smoothness requirement {\bf b.} is represented by
\begin{equation}\label{eq:spg2}
S(g)=\sum_{n_0+1\le i,j \le n_1-1}Qg_{i,j},
\end{equation}
where here $Q$ is the following quadratic operator:
\begin{equation}\label{eq:Q}
\begin{aligned}
Qg_{i,j}
&=[\Delta_{xx}g_{i,j}]^2+[\Delta_{yy}g_{i,j}]^2\\
&+\frac{1}{4}\{[\Delta_{xy}g_{i,j}]^2+
[\Delta_{xy}g_{i+1,j}]^2+[\Delta_{xy}g_{i,j+1}]^2+[\Delta_{xy}g_{i+1,j+1}]^2\}
\end{aligned}
\end{equation}
with $\Delta_{xx}g_{i,j}=g_{i-1,j}-2g_{i,j}+g_{i+1,j}$,
$\Delta_{yy}g_{i,j}=g_{i,j-1}-2g_{i,j}+g_{i,j+1}$ and
$\Delta_{xy}g_{i,j}=g_{i,j}-g_{i-1,j}-g_{i,j-1}+g_{i-1,j-1}$. We
note that the the functional $S(g)$ is related to the bi-harmonic
operator.

\bigskip
{\bf The bivariate approximation-extension algorithm:}

\begin{enumerate}
\item Find the model parameter $P$ for the data by minimizing $I_2$ .
\item Define the approximation on $D_0$, and its extension into $D$, as the sequence
$g=\{g_{i,j}\}_{n_0\le i,j \le n_1}$ minimizing the smooth
approximation functional
\begin{equation}\label{eq:F}
F(g)=S(g)+\mu E(g),\ \  g\in M.
\end{equation}
\end{enumerate}

\bigskip
{\bf Examples of bivariate approximation-extension using a model and
a smoothing functional:}

\medskip
Unlike the method of fitting exponentials, the algorithm based upon the smooth
approximation functional is completely linear. It is important to note that in
the bivariate case, even if we knew $g$ on $D_0$, a direct extension of $g$
into the larger domain $D$ by the model would be impossible. Hence, we must
solve here for $g$ on the entire mesh in the larger domain $D$. Of course, the
size of the linear system gets larger with the size $m$ of the model and the
size of the domain $D$.

In the following we demonstrate the performance of the proposed algorithm for
two test functions:
\begin{equation}\label{eq:f3}
f_3(x,y)=0.4cos(4(x+y))+0.6ysin(3(x-y))-(x-2)^2,
\end{equation}
\begin{equation}\label{eq:f4}
f_4(x,y)=x^2-y^3+2+x-y+20*exp(-(x-2)^2),
\end{equation}
In both cases the data is given on a square mesh, of mesh size $h=0.1$, in
$D_0=[0,4]^2$, with an added random noise in $[-0.2,0.2]$. The extension is
into a square mesh, of the same mesh size, in $D=[-2,6]^2$. A $4\times 4$
linear model is computed by (\ref{eq:ls2}), and the approximation-extension
$g$ is computed by minimizing the smooth approximation functional (\ref{eq:F})
with $\mu=100$. The linear system involves a sparse matrix of size $8900\times
8900$, and the solution is done by MINRES iterations.

For the function $f_3$ the resulting model coefficients are:
$$P=\begin{pmatrix}-0.0539&0.1662&-0.0964&-0.7606 \\
    0.3700&-0.0796&-0.3966&-0.4032\\
   -0.2253 &   0.0797  & -0.4422 &   0.2967\\
   -0.1989 &   0.1743  &  0.5024  &  1.0000
\end{pmatrix}$$
The noisy data of $f_3$ is shown in Figure \ref{fig:5_1} and the resulting
extension is shown in Figure \ref{fig:5_2}.

\vfill\eject
\begin{figure}[hb]
  \centering
  \includegraphics[width=3in]{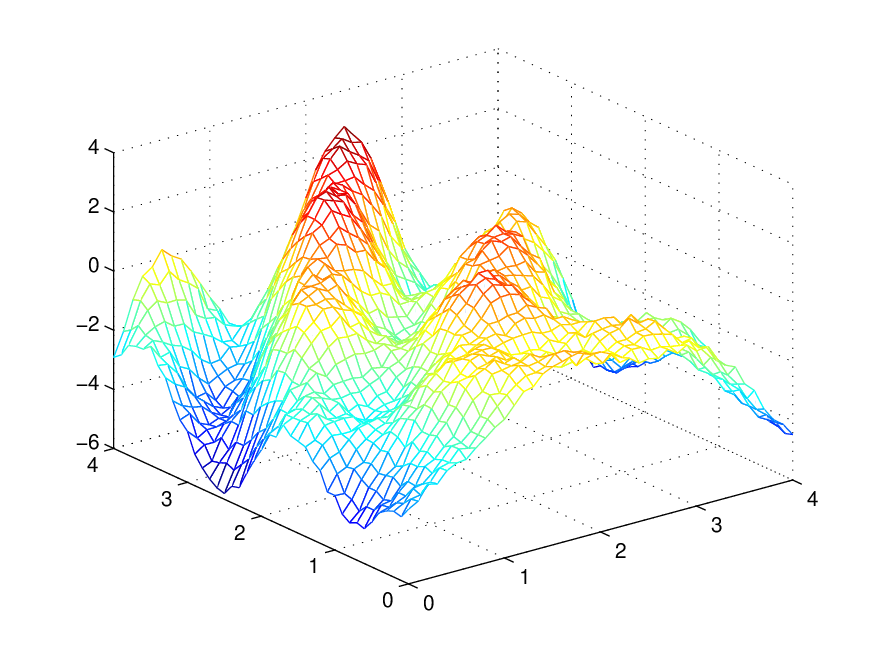}
  \caption[example1]{The non-noisy data of $f_3$ on $D_0=[0,4]^2$.}
  \label{fig:5_1}
\end{figure}

\begin{figure}[hb]
  \centering
  \includegraphics[width=3in]{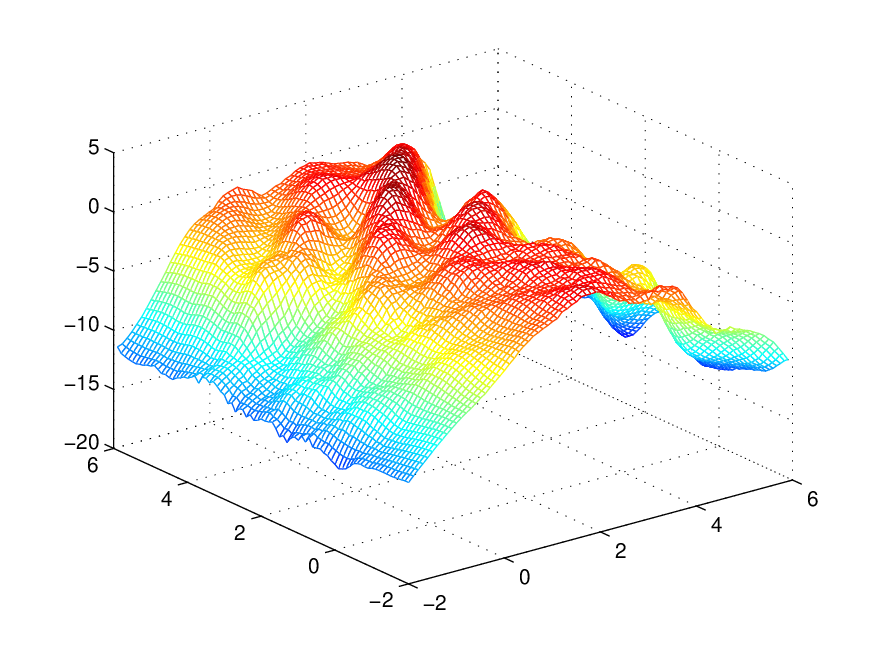}
  \caption[example2]{Approximation-extension of $f_3$ on $D=[-2,6]^2$.}
\label{fig:5_2}
\end{figure}

For the function $f_4$ the resulting model coefficients are:
$$P=\begin{pmatrix}  0.1505  &  0.0772 &   0.1560 &  -0.3483\\
    0.0480 &   0.2556 &   0.1838 &  -0.0669\\
   -0.1204 &  -0.0840 &  -0.0200 &  -0.4786\\
   -0.1863 &  -0.2081 &  -0.3254 &   1.0000
\end{pmatrix}$$

The noisy data of $f_4$ is shown in Figure \ref{fig:5_3} and the resulting
extension is shown in Figure \ref{fig:5_4}.

\vfill\eject
\begin{figure}[hb]
  \centering
  \includegraphics[width=3in]{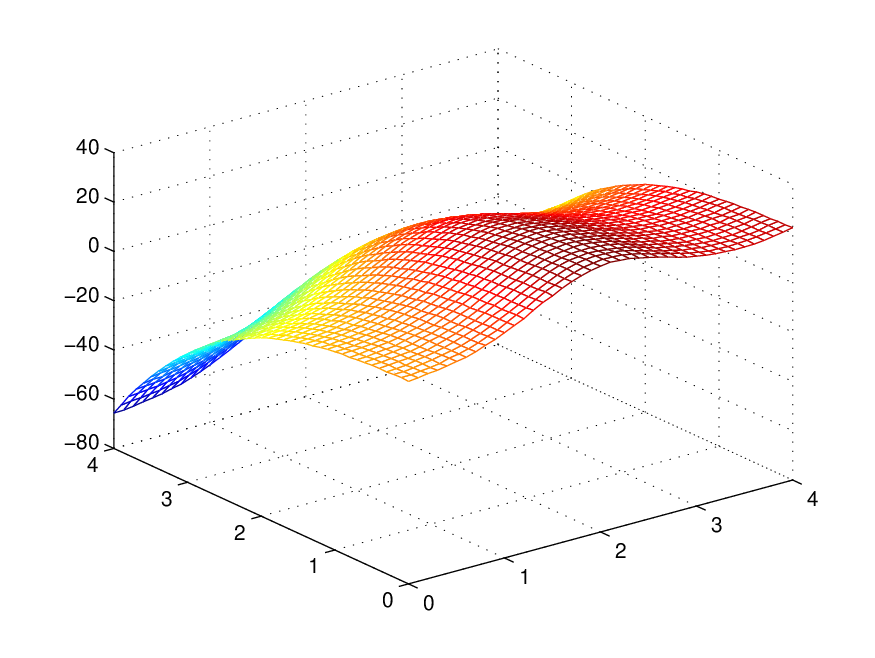}
  \caption[example1]{The non-noisy data of $f_4$ on $D_0=[0,4]^2$.}
  \label{fig:5_3}
\end{figure}

\begin{figure}[hb]
  \centering
  \includegraphics[width=3in]{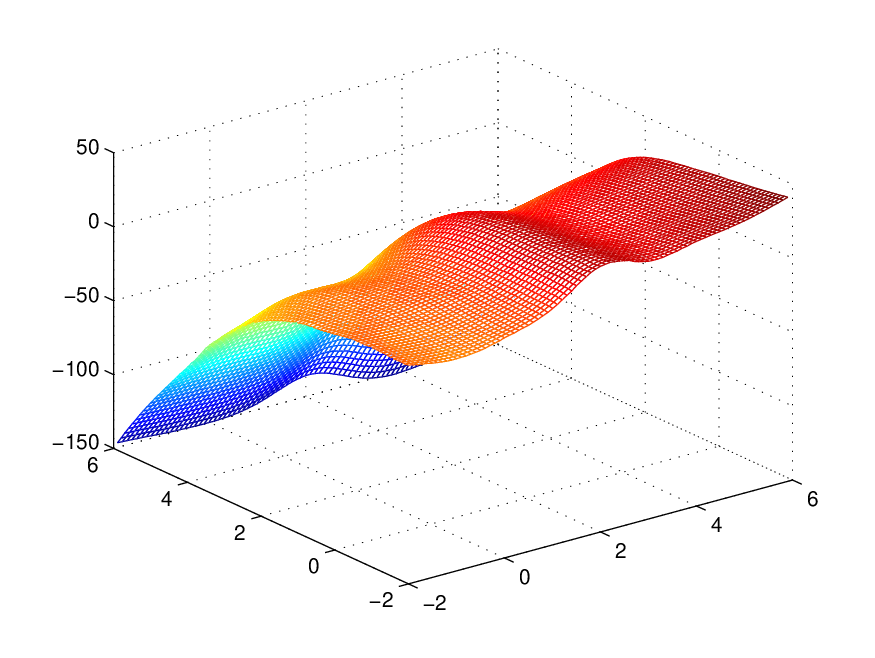}
  \caption[example2]{Approximation-extension of $f_4$ on $D=[-2,6]^2$.}
\label{fig:5_4}
\end{figure}

\vfill\eject
\section{Efficient approximation-extension using spline basis functions model}

The smooth approximation method for approximation-extension discussed above has two major drawbacks: first, the computational complexity increases significantly in the bivariate case; second, selecting an appropriate balancing parameter 
$\mu$ between the approximation and smoothing functionals in equation (\ref{eq:F}) is challenging.

We now introduce a more efficient alternative method that eliminates the need for a balancing parameter. However, this approach is less suited for models with varying coefficients or more complex, general models. The efficiency of this method stems from the following simple observation:

\begin{prop}
\label{p:Prop1} Let $P=\{p_{k,\ell}\}_{(k,\ell)\in K}$ represent a
linear model with constant coefficients, where $K$ is a finite
subset of $\ZZ^2$, and let $\phi$ be a function of compact support
in $\RR^2$. Consider
\begin{equation}\label{eq:gphi}
g(x,y)=\sum_{(i,j)\in \ZZ^2}c_{i,j}\phi(x-i,y-j),\ \ (x,y)\in\RR^2.
\end{equation}
where $\{c_{i,j}\}_{(i,j)\in \ZZ^2}$ is a bi-infinite sequence satisfying the
model, i.e.,
\begin{equation}\label{eq:cinM}
\sum_{(k,\ell)\in K}p_{k,\ell}c_{i+k,j+\ell}=0\ \ \forall (i,j)\in\ZZ^2.
\end{equation}
Then, the function $g$ satisfies the model, i.e.,
\begin{equation}\label{eq:ginM}
\sum_{(k,\ell)\in K}p_{k,\ell}g(x+k,y+\ell)=0\ \ \forall (x,y)\in\RR^2.
\end{equation}
\end{prop}
\begin{proof}
Using the compact support of $K$ and of $\phi$,
\begin{equation}\label{eq:sumsum}
\begin{aligned}
\sum_{(k,\ell)\in K}p_{k,\ell}g(x+k,y+\ell) &=\sum_{(k,\ell)\in
K}p_{k,\ell}\sum_{(i,j)\in \ZZ^2}c_{i,j}\phi(x-i+k,y-j+\ell)\\
&=\sum_{(k,\ell)\in
K}p_{k,\ell}\sum_{(r,s)\in \ZZ^2}c_{r+k,s+\ell}\phi(x-r,y-s)\\
&=\sum_{(r,s)\in \ZZ^2}\phi(x-r,y-s)\sum_{(k,\ell)\in
K}p_{k,\ell}c_{r+k,s+\ell}\\
&=0
\end{aligned}
\end{equation}
\end{proof}
\begin{prop}
\label{p:Prop2} Let $P=\{p_{k,\ell}\}_{(k,\ell)\in K}$ represent a
linear model with constant coefficients, where $K$ is a finite
subset of $\ZZ^2$, and let $\phi$ be a function of compact support
in $\RR^2$ such that its integer shifts are linearly independent.
Consider
\begin{equation}\label{eq:gphi2}
g(x,y)=\sum_{(i,j)\in \ZZ^2}c_{i,j}\phi(x-i,y-j),\ \ (x,y)\in\RR^2.
\end{equation}
If the function $g$ satisfies the model, i.e., (\ref{eq:ginM})
holds, then $\{c_{i,j}\}_{(i,j)\in \ZZ^2}$ satisfies the model,
i.e., (\ref{eq:cinM}) holds.
\end{prop}
\begin{proof}
The proof follows directly from (\ref{eq:sumsum}) using the linear
independence of $\{\phi(x-i,y-j)\}$.
\end{proof}
{\bf Remark:} The propositions are formulated for the bivariate
case, but the results hold in any dimension.

\medskip
Given Proposition \ref{p:Prop1} we can easily generate many
functions that satisfy a given model by sums of integer shifts of
any function $\phi$. We want to use smooth basis functions that can form a basis for approximating the data. A suitable choice for this purpose is to define $\phi$ as a B-spline
(e.g., tensor product) with equidistant knots and a mesh size $d$. The mesh
size $d$ should be selected so that $d$-shifts of the B-spline
provide a good approximation space to the function $f$ that we aim to
extend. With this approach, a spline approximation to $f$ will approximately satisfy the same model as $f$
and, due to Proposition
\ref{p:Prop2}, the B-spline coefficients will also approximately
satisfy the same model. The choice of B-splines is natural because of the smoothness functionals used above, that are related to
splines. In our tests, we have used cubic B-splines and their tensor
products. Our next goal is to develop an efficient method for generating the space of splines that satisfy a given model.

\subsection{Model-Spline basis functions for approximation-extension}

Let us start with the univariate case. After selecting the generating function $\phi$ as a cubic B-spline, and the parameter $d$, we would like to
construct a basis for all the splines with equidistant knots, of mesh size
$d$, that satisfy the model. Here again, we assume, w.l.o.g., that $d=1$. We
consider coefficients satisfying a univariate model $M$ of order $m$ of the
form
\begin{equation}\label{eq:cinM1}
c_{i+m+1}=\sum_{k=1}^mp_kc_{i+k},\ \ i\in\ZZ.
\end{equation}
For convenience we denote by $\{c_i\}\in M$ a sequence satisfyng the
model. Viewing (\ref{eq:cinM1}) as a prediction model, any vector of
$m$ initial values $\{c_i\}_{i=r}^{r+m-1}$ generates a sequence
satisfying (\ref{eq:cinM1}) for $i\ge r$. A basis for all the
sequences $\{c_i\}\in M$ for $i\ge r$ may be obtained by using $m$
independent initial vectors, e.g., the standard basis vectors
$e^{(j)}=\{\delta_{i,j}\}_{i=1}^m$, $j=1,...,m$. For an
approximation in $[0,N]$, using the standard cubic B-spline
representation, we need the coefficients $\{c_i\}_{i=-1}^{N+1}$. We
thus generate a basis of $m$ sequences $\{c^{(j)}_i\}\in M$, for
$i\ge -1$, by $m$ independent vectors of initial condition
$\{c^{(j)}_i\}_{i=-1}^{m-2}=e^{(j)}$, $j=1,...,m$. The $m$ spline
functions corresponding to these $m$ sequences are
\begin{equation}\label{eq:S1}
S_j(x)=\sum_{i=-1}^{N+1}c^{(j)}_iB(x-i),\ \ \ j=1,...,m.
\end{equation}
By Proposition \ref{p:Prop1}, any $g\in span\{S_j\}_{j=1}^m$ satisfies the
model $M$. We call such $g$ a model-spline. Thus, the task for the approximation-extension procedure reduces to finding  $g\in span\{S_j\}_{j=1}^m$
that best approximates the given data. To demonstrate the method we go back
to the previous examples, with noisy data for the test functions $f_1$ and
$f_2$, with $m=6$. First, we present plots the basis functions
$\{S_j\}_{j=1}^6$, corresponding to the model found for $f_1$ in Figure
\ref{fig:basisf1}, and for $f_2$ in Figure \ref{fig:basisf2}. One can already
see the behavior of $f_1$ and of $f_2$ living in their corresponding basis
functions. Figure \ref{fig:bsae1} below illustrates the approximation-extension process using these basis functions.

\medskip
{\bf Remarks:}
\begin{enumerate}
\item Any shift of the basis functions $\{S_j\}_{j=1}^m$ also
satisfy the same model. We can use such shift in order to augment our basis
for the approximation-extension algorithm.
\item This approach is both faster in practice and free of adjustable parameters. The smoothness is inherently determined by the choice of the function
$\phi$.
\end{enumerate}

\begin{figure}[hb]
  \centering
  \includegraphics[width=3in]{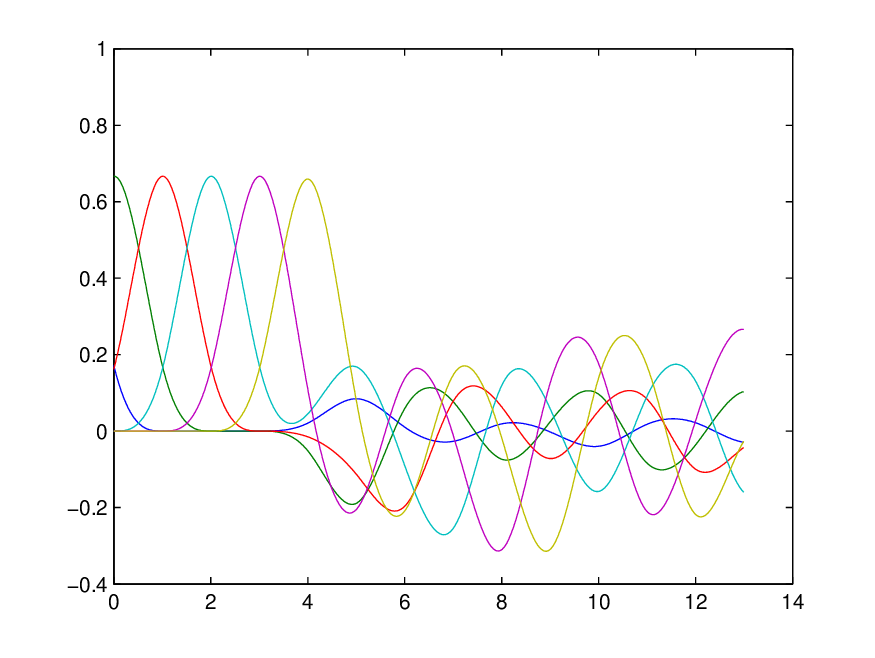}
  \caption[example1]{$\{S_j\}_{j=1}^6$ for the model found for $f_1$.}
  \label{fig:basisf1}
\end{figure}

\begin{figure}[hb]
  \centering
  \includegraphics[width=3in]{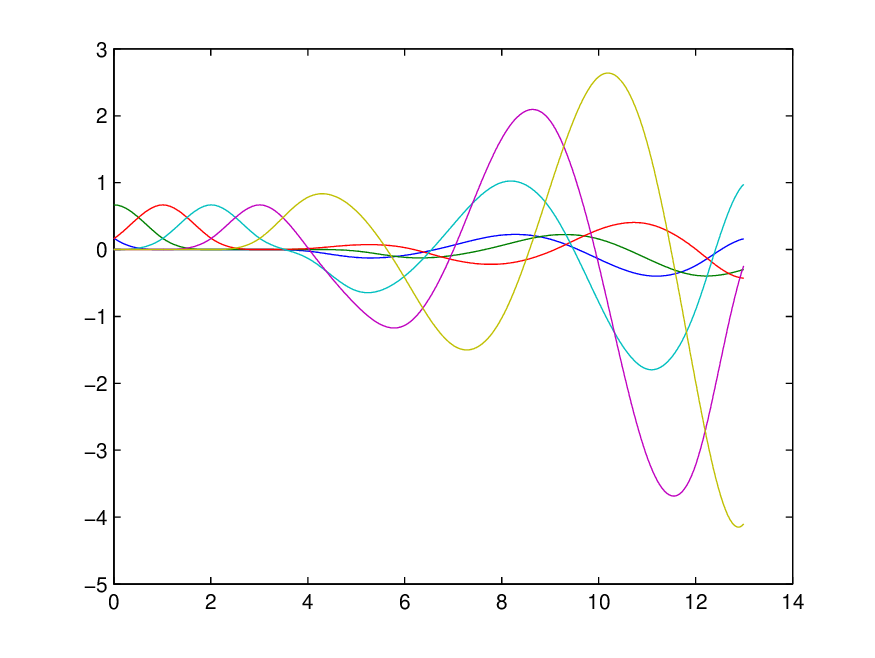}
  \caption[example2]{$\{S_j\}_{j=1}^6$ for the model found for $f_2$}
\label{fig:basisf2}
\end{figure}

\vfill\eject
\begin{figure}[hb]
  \centering
  \includegraphics[width=3in]{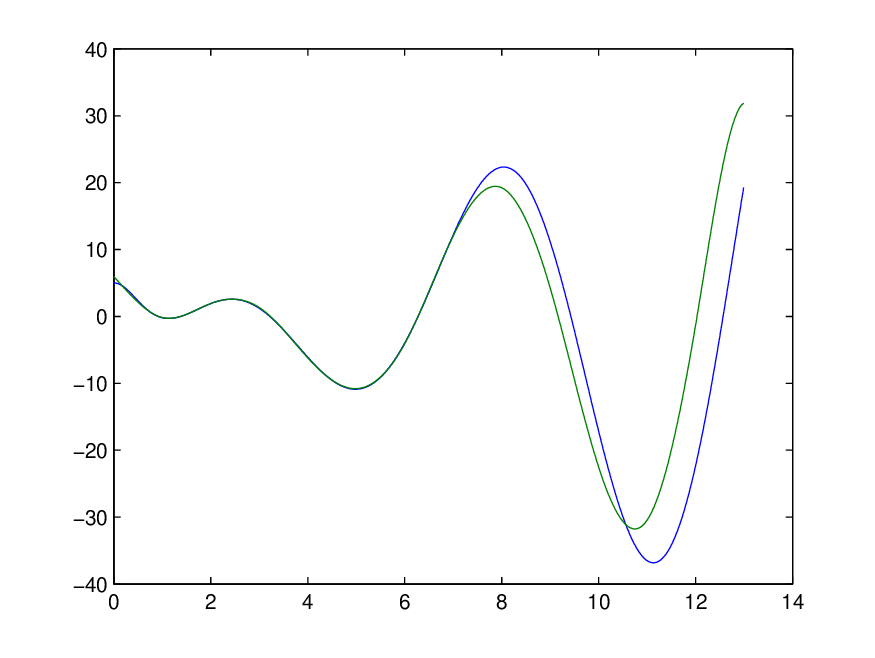}
  \caption[example1]{Approximation-extension of $f_2$ using the model-spline basis.}
  \label{fig:bsae1}
\end{figure}

\medskip
{\bf Remarks:}
\begin{enumerate}
\item Any shift of the basis functions $\{S_j\}_{j=1}^m$ also
satisfy the same model. We can use such shift in order to augment our basis
for the approximation-extension algorithm.
\item The above approach is both faster in application and is parameters'
free. The smoothness is determined directly by the choice of the function
$\phi$.
\end{enumerate}

\bigskip
{\bf Model-basis functions for the bivariate case}

\bigskip
The main motivation for using the basis functions' approach is the high
complexity involved is the application of the smoothing approach in 2-D.
Consider an $m\times m$ order bivariate model of the form (\ref{eq:model2D}).
W.l.o.g., we would like to form a basis for all the tensor-product bi-cubic
splines, with integer knots, that satisfy the model on $[0,N]^2$. Any such
bi-cubic spline can be written as
\begin{equation}\label{eq:S2}
S(x,y)=\sum_{i,j=-1}^{N+1}c_{i,j}B(x-i)B(y-j).
\end{equation}
According to Propositions \ref{p:Prop1} and \ref{p:Prop2}, the coefficients
$\{c_{i,j}\}$ should satisfy the model. I.e.,
\begin{equation}\label{eq:modelcij}
\sum_{1\le k,\ell\le m} p_{k,\ell}c_{i+k,j+\ell}=0,\ \ with \ \ p_{m,m}=1.
\end{equation}
Let $L=\{(i,j)|-1\le i\vee j\le m-1\}$ denote the $m-1$ layers of
indices along the bottom and the left boundary of the set of indices
$\{(i,j)\}_{i,j=-1}^{N+1}$. Given any boundary values for
$\{c_{i,j}\}_{(i,j)\in L}$, we can use the above model to fill out
the rest of the coefficients by:
\begin{equation}\label{eq:cijpred}
c_{i+m,j+m}=-\sum_{1\le k,\ell\le m;\ (k,\ell)\not= (m,m)}
p_{k,\ell}c_{i+k,j+\ell},\ \ i,j=0,...,N-m+1.
\end{equation}
Hence, the space of all the tensor-product bi-cubic splines, with integer
knots, that satisfy the model on $[0,N]^2$ is of dimension $\#(L)$. A basis
to this space can be generated, as in the univariate case, by taking a basis
for the space $V_L=\{c_{i,j}|(i,j)\in L\}$, extending each vector in the basis
by (\ref{eq:cijpred}), and using the resulting coefficients to define a basis
function by (\ref{eq:S2}).

In the following we go back to the bivariate example with noisy data for
$f_3$, now with a $6\times 6$ model. In this case the dimension of the
approximating space of spline functions satisfying the model is 85. To get a
feeling of the basis functions involved we plot one of them in Figure
\ref{fig:base65}. In Figure \ref{fig:f3data} we see the given data, in Figure
\ref{fig:apprf3} the approximation to the data by the spline basis functions,
and in Figure \ref{fig:f3extended} the extension into a larger domain.

\vfill\eject
\begin{figure}[hb]
  \centering
  \includegraphics[width=3in]{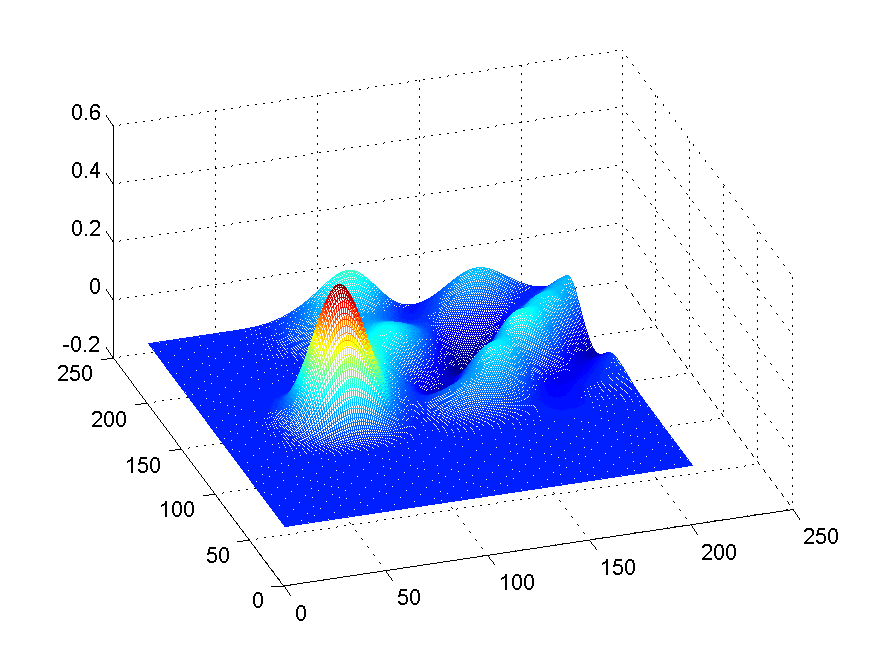}
  \caption[example1]{One of the model-spline basis functions for reconstructing $f_3$.}
  \label{fig:base65}
\end{figure}

\vfill\eject
\begin{figure}[hb]
  \centering
  \includegraphics[width=3in]{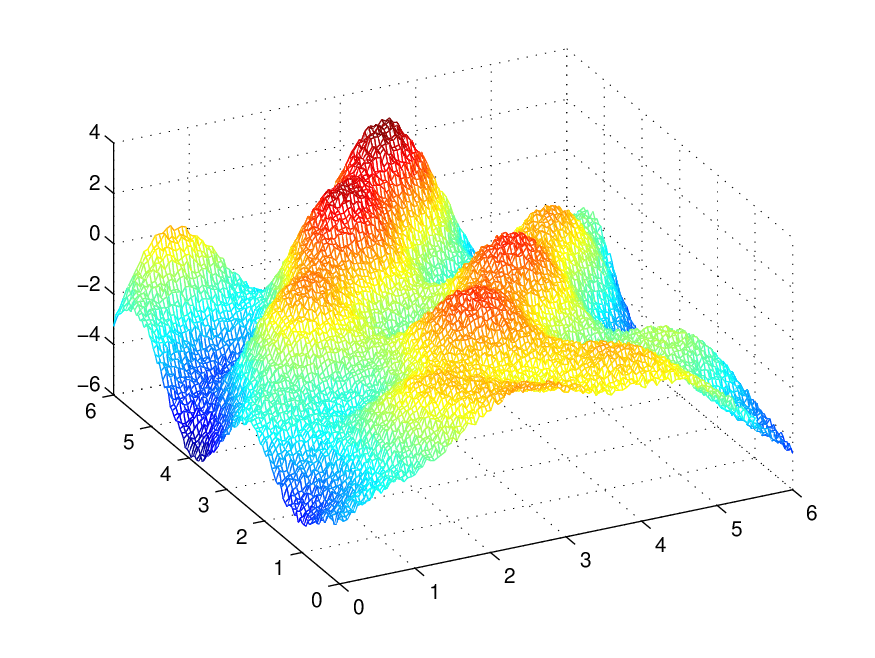}
  \caption[example1]{The noisy data for $f_3$.}
  \label{fig:f3data}
\end{figure}

\begin{figure}[hb]
  \centering
  \includegraphics[width=3in]{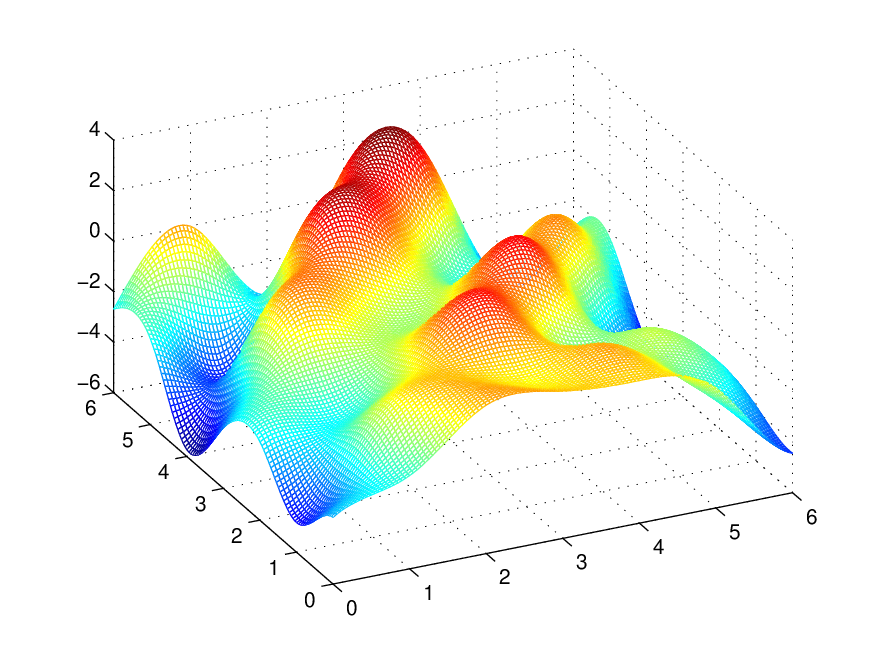}
  \caption[example2]{The approximation of $f_3$ by model-splines $f_3$}
\label{fig:apprf3}
\end{figure}

\begin{figure}[hb]
  \centering
  \includegraphics[width=3in]{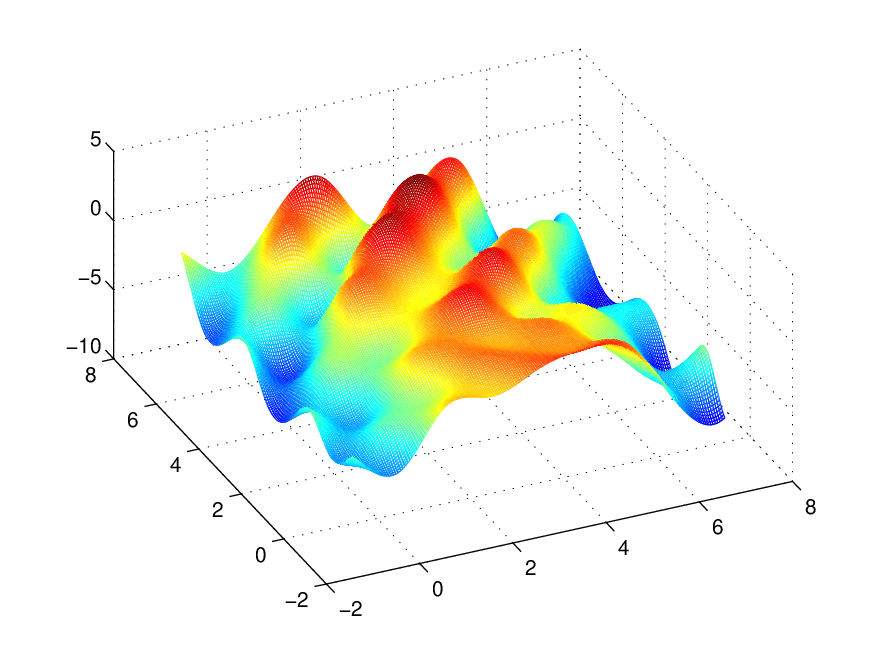}
  \caption[example1]{The extension of $f_3$ by model splines.}
  \label{fig:f3extended}
\end{figure}

\subsection{Interpolation between models}
In all the above examples we have demonstrated methods for functions'
extensions that preserve their behavior. Now we consider the possibility of
generating a smooth extension of a function that blends its behavior into
another desired behavior. This seems a whole project by itself, but with the
tools presented in this work we can already present a basic method and an
example that illustrate the potential of this direction. We consider the
univariate case, and start by fitting a model to a given data on $[a,b]$.
Denoting this model by $M1$, we would like to approximate the data by a smooth
function on $[a,b]$, and extend this approximation into $[b,c]$, so that the
behavior model will change smoothly from $M_1$ on $[a,b]$ to another behavior
$M_2$ at $c$. The simple idea is to define a linear model $M$ on $[a,c]$ that
is a blending of the two models.

We assume that both models are linear and of the same order $m$. The simple
idea is to define a linear model $M$ on $[a,c]$ that is a blending of the two
models. In Figure \ref{fig:cos2x} we depict the noisy data of $f(x)=cos(2x)$
on $[0,5]$. Computing a linear model of order 4, with constant coefficients,
for this data yield the coefficients of $M_1$:
$P=\{0.0803,0.6555,-0.4325,-0.2470,-1.0000\}$. We have chosen $M_2$ to be the
model for the function $f(x)=e^{-2x}$. The corresponding model coefficients
are: $P=\{0,0,0.0178,0.0024,-1.0000$. We now define a linear model with linear
coefficients (as in (\ref{eq:pred1}), by a linear interpolation between the
corresponding coefficients of the two models, such that it agrees with $M_1$
at $x=0$ and with $M_2$ at $x=8$. The resulting model is of the form
(\ref{eq:pred1}) with $u(x)=x$, and we can use it as in Section \ref{sec2} to
define the approximation extension of the data. The result is shown in Figure
\ref{fig:blend}.

\vfill\eject
\begin{figure}[hb]
  \centering
  \includegraphics[width=3in]{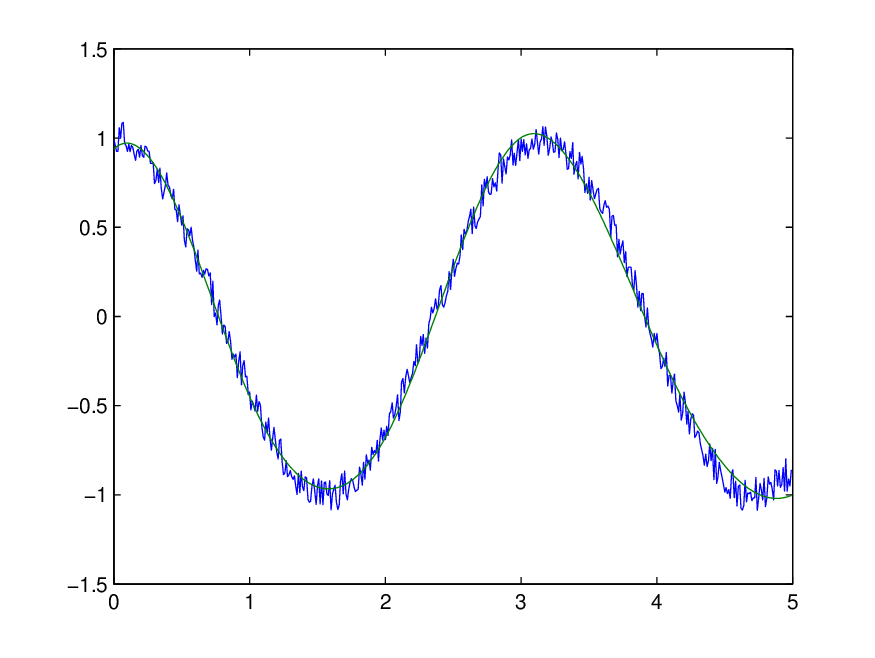}
  \caption[example2]{The approximation of the data of $cos(2x)$ by the blended mpdel.}
\label{fig:cos2x}
\end{figure}

\begin{figure}[hb]
  \centering
  \includegraphics[width=3in]{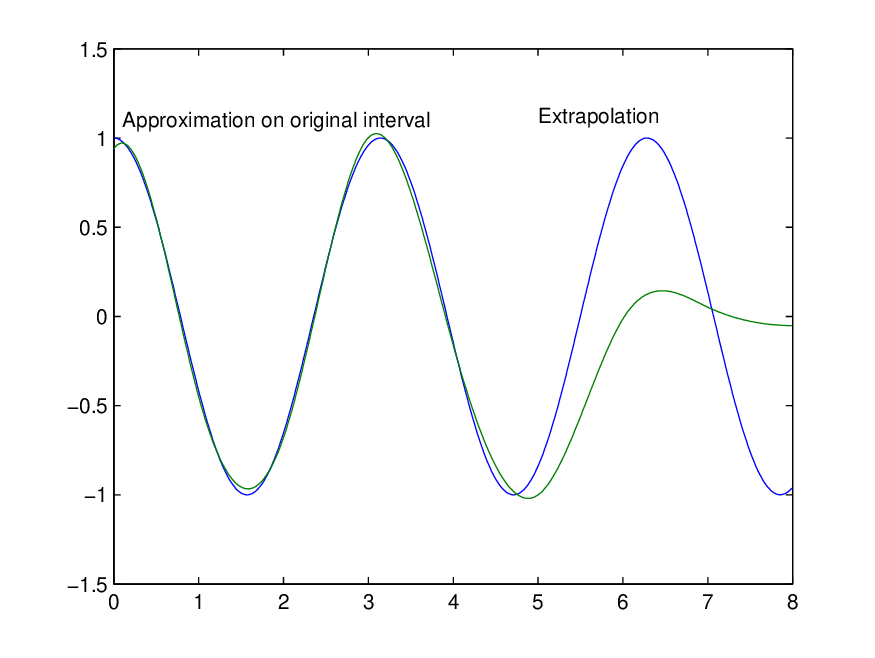}
  \caption[example1]{The approximation of $cos(2x)$ extended to decay as $e^{-2x}$.}
  \label{fig:blend}
\end{figure}

\bibliographystyle{amsplain}
\bibliography{zmin2}

\end{document}